\documentclass[a4paper,10pt]{article}
\usepackage[latin1]{inputenc}
\usepackage[T1]{fontenc}
\usepackage{amsmath}
\usepackage{amsfonts}
\usepackage{amstext}
\usepackage{amsthm}
\usepackage[all]{xy}
\usepackage{geometry}
\usepackage{srcltx}
\usepackage{graphics}
\usepackage{epsfig}
\usepackage{subfigure}
\usepackage{slashed}
\usepackage{color}
\usepackage{amssymb}
\usepackage{srcltx}
\newtheorem{theorem}{Theorem}[section]
\newtheorem{lemma}[theorem]{Lemma}

\newtheorem{rem}[theorem]{Remark}
\newtheorem{prop}[theorem]{Proposition}
\newtheorem{cor}[theorem]{Corollary}

\DeclareMathOperator{\im}{im}

\DeclareMathOperator{\sfl}{sf}
\DeclareMathOperator{\diag}{diag}

\DeclareMathOperator{\Sp}{Sp}
\DeclareMathOperator{\gra}{graph}

\title{The Maslov Index and the Spectral Flow - revisited}
 
\author{Marek Izydorek, Joanna Janczewska and Nils Waterstraat}

\begin{document}
\date{}
\maketitle

\footnotetext[2]{This work was supported by the grant BEETHOVEN2 of the National Science Centre, Poland, no. 2016/23/G/ST1/04081.}

\begin{abstract}
\noindent
We give an elementary proof of a celebrated theorem of Cappell, Lee and Miller which relates the Maslov index of a pair of paths of Lagrangian subspaces to the spectral flow of an associated path of selfadjoint first-order operators. We particularly pay attention to the continuity of the latter path of operators, where we consider the gap-metric on the set of all closed operators on a Hilbert space. Finally, we obtain from Cappell, Lee and Miller's theorem a spectral flow formula for linear Hamiltonian systems which generalises a recent result of Hu and Portaluri.
\end{abstract}

\section{Introduction}
Let $\langle\cdot,\cdot\rangle$ be the Euclidean scalar product on $\mathbb{R}^{2n}$ and $\omega_0(\cdot,\cdot)=\langle J\cdot,\cdot\rangle$ the standard symplectic form, where 

\begin{align}\label{J}
J=\begin{pmatrix}
0&-I_n\\
I_n&0
\end{pmatrix}
\end{align}
and $I_n$ denotes the identity matrix. Let us recall that an $n$-dimensional subspace $L\subset\mathbb{R}^{2n}$ is called Lagrangian if the restriction of $\omega_0$ to $L\times L$ vanishes. The set $\Lambda(n)$ of all Lagrangian subspaces in $\mathbb{R}^{2n}$ is called the \textit{Lagrangian Grassmannian}. It can be regarded as a submanifold of the Grassmannian $G_n(\mathbb{R}^{2n})$ and so it has a canonical topology. In what follows, we denote by $I$ the unit interval $[0,1]$. The Maslov index $\mu_{Mas}(\gamma_1,\gamma_2)$ assigns to any pair of paths $\gamma_1, \gamma_2:I\rightarrow\Lambda(n)$ an integer which, roughly speaking, is the total number of non-trivial intersections of the Lagrangian spaces $\gamma_1(\lambda)$ and $\gamma_2(\lambda)$ whilst the parameter $\lambda$ travels along the interval $I$. There are several different approaches to the Maslov index and here we just want to mention \cite{Maslov}, \cite{Bott}, \cite{Duistermaat}, \cite{Sternberg}, \cite{RobbinMaslov} and \cite{Zehnder}, which is far from being exhaustive. Cappell, Lee and Miller introduced in \cite{Cappell} four different ways to define the Maslov index and showed that they are all equivalent. They first construct the Maslov index geometrically by using a stratification of $\Lambda(n)$ and intersection theory from differential topology following \cite{Sternberg}. Their approach also yields a uniqueness theorem for the Maslov index characterising this invariant uniquely by six axioms. The uniqueness theorem is then used to show that the Maslov index can alternatively be defined by determinant line bundles, $\eta$-invariants and the spectral flow, respectively.\\
In this paper we focus on the latter invariant and aim to give a more elementary proof of the equality of the Maslov index and the spectral flow of a path of operators as introduced by Cappell, Lee and Miller in \cite{Cappell}. Let us first recall that the spectral flow is a homotopy invariant for paths of selfadjoint Fredholm operators that was invented by Atiyah, Patodi and Singer in \cite{APS}, and since then has been used in various different settings (see e.g. \cite[\S 5.2]{Fredholm}). The spectrum of a selfadjoint Fredholm operator consists only of eigenvalues of finite multiplicity in a neighbourhood of $0\in\mathbb{R}$ and, roughly speaking, the spectral flow of a path of such operators is the net number of eigenvalues crossing $0$ whilst the parameter of the path travels along the interval.\\
Let us now consider for a pair of paths $(\gamma_1,\gamma_2)$ in $\Lambda(n)$ the differential operators

\begin{align}\label{ops-def}
\mathcal{A}_\lambda:\mathcal{D}(\mathcal{A}_\lambda)\subset L^2(I,\mathbb{R}^{2n})\rightarrow L^2(I,\mathbb{R}^{2n}),\quad (\mathcal{A}_\lambda u)(t)=Ju'(t),
\end{align}
where

\begin{align}\label{ops-def-domain}
\mathcal{D}(\mathcal{A}_\lambda)=\{u\in H^1(I,\mathbb{R}^{2n}):\, u(0)\in\gamma_1(\lambda), u(1)\in\gamma_2(\lambda)\}.
\end{align}
By an elementary computation, $\mathcal{A}_\lambda$ is symmetric, and it is also not difficult to see that it actually is a selfadjoint Fredholm operator. Note that the kernel of $\mathcal{A}_\lambda$ is isomorphic to $\gamma_1(\lambda)\cap\gamma_2(\lambda)$, which suggests that the spectral flow of the path $\mathcal{A}=\{\mathcal{A}_\lambda\}_{\lambda\in I}$ is related to the Maslov index of the pair $(\gamma_1,\gamma_2)$. As we already mentioned above, their equality is one of the main achievements of \cite{Cappell}. However, before we formulate this as a theorem, we want to highlight a further issue related to this problem.\\
Above, we have spoken about paths of differential operators and so tacitly assumed continuity. Note that the family \eqref{ops-def} has the non-constant domains \eqref{ops-def-domain} and so continuity is a non-trivial problem. There are different metrics on spaces of unbounded selfadjoint Fredholm operators on a Hilbert space $H$ and we recommend \cite{Lesch} for an exhaustive discussion (see also \cite{Wahl}). A classical approach is to transform unbounded selfadjoint operators $T$ by functional calculus to the bounded selfadjoint operators 

\begin{align}\label{Riesz}
(I_H+T^2)^{-\frac{1}{2}}\in\mathcal{L}(H),
\end{align}
and to use the operator norm on $\mathcal{L}(H)$ for introducing a distance between unbounded operators. Actually, Atiyah, Patodi and Singer defined the spectral flow in \cite{APS} for bounded selfadjoint Fredholm operators and applied it to paths of differential operators by using \eqref{Riesz}. However, checking continuity along these lines is tedious, if possible at all (see e.g. \cite{Nicolaescu}), and it seems that the continuity of families of unbounded operators has sometimes been ignored in the literature.\\
Every (generally unbounded) selfadjoint operator on a Hilbert space is closed, and there is a canonical metric on the set of all closed operators which is called the gap-metric (see \S IV.2 in Kato's monograph \cite{Kato}). It was shown in \cite{Nicolaescu} (see also \cite[Prop. 2.2]{Lesch}) that every path of selfadjoint Fredholm operators that is mapped to a continuous path of bounded operators under \eqref{Riesz} is also continuous with respect to the gap-metric. Finally, Booss-Bavnbek, Lesch and Phillips constructed in \cite{UnbSpecFlow} the spectral flow for paths of selfadjoint Fredholm operators in this more general setting. The main result of this paper now reads as follows (see \cite[Thm. 0.4]{Cappell}).

\begin{theorem}\label{main}
If $(\gamma_1,\gamma_2)$ is a pair of paths in $\Lambda(n)$, then the family of differential operators \eqref{ops-def} is continuous with respect to the gap-metric and 

\[\sfl(\mathcal{A})=\mu_{Mas}(\gamma_1,\gamma_2).\]
\end{theorem}
\noindent
Let us make a few comments on our proof. Firstly, we want to emphasise that we prove the gap-continuity of the family \eqref{ops-def} from first principles just by elementary estimates and standard facts about orthogonal projections that can all be found in the monograph \cite{Kato}. Secondly, our proof of the spectral flow formula in Theorem \ref{main} is surprisingly simple. We assume at first that 

\begin{align}\label{introduction-transversal}
\gamma_1(0)\cap\gamma_2(0)=\gamma_1(1)\cap\gamma_2(1)=\{0\}
\end{align}
and show that the Maslov index can be characterised in this case by three axioms. This uniqueness theorem needs nothing else than the elementary properties of the Maslov index and the fact that the fundamental group of $\Lambda(n)$ is infinitely cyclic, which was known already from Arnold's classical paper \cite{Maslov}. Two of our axioms are trivially satisfied for the spectral flow of \eqref{ops-def}, and the remaining one only requires the computation of the spectra of two simple examples of differential operators as in \eqref{ops-def}. The general case when \eqref{introduction-transversal} is not assumed, can easily be obtained from the previous case by a simple conjugation by a path of invertible operators. After a brief recapitulation of the Maslov index in Section \ref{section-Maslov}, and the gap-metric and spectral flow in Section \ref{section-spectralflow}, we explain all this in detail in Section \ref{section-proof} where we prove Theorem \ref{main}. Throughout the paper, we aim our presentation to be rather self-contained, and we will just use some well-known facts from Kato \cite{Kato}.\\
Finally, we review a recent spectral flow formula for linear Hamiltonian systems by Hu and Portaluri from \cite{Hu}, which they call \textit{a new index theory on bounded domains}. Firstly, we note that the considered families of Hamiltonian systems are continuous with respect to the gap-metric, which follows easily from our approach to Cappell, Lee and Miller's Theorem. Secondly, we obtain a spectral flow formula in this setting by a conjugation from Cappell, Lee and Miller, and we explain that our result actually is a generalisation of Hu and Portaluri's Theorem.


\section{Maslov Index and Spectral Flow - a brief recap}

\subsection{The Maslov Index}\label{section-Maslov}
The aim of this section is to briefly recall the definition of the Maslov index, where we follow \cite{PiccioneBook}.\\
Let $\Sp(2n,\mathbb{R})$ denote the group of symplectic matrices on $\mathbb{R}^{2n}$, i.e., those $A\in M(2n,\mathbb{R})$ satisfying $A^TJA=J$ or, alternatively, which preserve $\omega_0$. If we identify $\mathbb{R}^{2n}$ with $\mathbb{C}^n$ by $(x_1,\ldots,x_{2n})\mapsto (x_1,\ldots,x_n)+i(x_{n+1},\ldots, x_{2n})$ then the standard hermitian scalar product on $\mathbb{C}^n$ is 

\[\langle x,y\rangle_{\mathbb{C}}=\langle x,y\rangle-i\omega_0(x,y).\] 
Hence each unitary matrix $U\in U(n)$ preserves $\omega_0$ and so we can regard $U(n)$ as a subset of $\Sp(2n,\mathbb{R})$. Also, the orthogonal matrices $O(n)$ can be seen as a subgroup of $U(n)$ by complexification. Then $O(n)$ consists exactly of those $A\in U(n)$ which leave $\mathbb{R}^n\times\{0\}$ invariant.\\
Obviously, $AL\in\Lambda(n)$ if $L\in\Lambda(n)$ and $A\in\Sp(2n,\mathbb{R})$, and it can be shown that the restriction of this action to $U(n)\times \Lambda(n)\rightarrow\Lambda(n)$ is transitive. As the stabiliser subgroup of $\mathbb{R}^n\times\{0\}\in\Lambda(n)$ is $O(n)$, we see that there is a diffeomorphism

\begin{align}\label{diffeoU}
U(n)/O(n)\simeq\Lambda(n),\quad A\mapsto A(\mathbb{R}^n\times\{0\}).
\end{align}
Let us now consider the map $d:U(n)\rightarrow S^1$, $d(A)=\det^2(A)$, which descends to the quotient by

\[\overline{d}:U(n)/O(n)\rightarrow S^1,\quad A\cdot O(n)\mapsto \det{^2}(A).\]
Note that
\[\ker(d)/O(n)\hookrightarrow U(n)/O(n)\xrightarrow{\overline{d}} S^1\]
is a fibre bundle, and it is not difficult to see that $\ker(d)/O(n)\simeq SU(n)/SO(n)$, where the latter space is simply connected. It follows from the long exact sequence of a fibre bundle that the induced map 

\[\overline{d}_\ast:\pi_1(U(n)/O(n))\rightarrow\pi_1(S^1)\cong\mathbb{Z}\]
is an isomorphism. Consequently, we obtain from \eqref{diffeoU} an isomorphism

\[\mu_{Mas}:\pi_1(\Lambda(n))\rightarrow\mathbb{Z},\]
which is the \textit{Maslov index} for closed paths in $\Lambda(n)$. Roughly speaking, given an arbitrary $L_0\in\Lambda(n)$, the Maslov index counts the total number of intersections of a loop in $\Lambda(n)$ with $L_0$. This is independent of the particular choice of $L_0$, which however is no longer the case if we extend the definition to non closed paths in $\Lambda(n)$ as follows.\\
We fix $L_0\in\Lambda(n)$ and note at first that $L_0$ yields a stratification 

\[\Lambda(n)=\bigcup^n_{k=0}\Lambda_k(L_0),\]
where

\[\Lambda_k(L_0)=\{L\in\Lambda(n):\, \dim(L\cap L_0)=k\}.\]    
From the fact that $\Lambda_0(L_0)$ is contractible (see e.g. \cite[Rem. 2.5.3]{PiccioneBook}) and the long exact sequence of homology, we see that the inclusion induces an isomorphism

\[H_1(\Lambda(n))\rightarrow H_1(\Lambda(n),\Lambda_0(L_0)).\]
Also, as $\pi_1(\Lambda(n))$ is abelian, $H_1(\Lambda(n))$ is isomorphic to $\pi_1(\Lambda(n))$ and so we obtain a sequence of isomorphisms

\begin{align}\label{Maslovsequence}
H_1(\Lambda(n),\Lambda_0(L_0))\rightarrow H_1(\Lambda(n))\rightarrow \pi_1(\Lambda(n))\rightarrow\pi_1(U(n)/O(n))\rightarrow\mathbb{Z}.
\end{align}
Finally, every path in $\Lambda(n)$ having endpoints in $\Lambda_0(L_0)$ canonically yields an element in\linebreak $H_1(\Lambda(n),\Lambda_0(L_0))$. The \textit{Maslov index} of the path is the integer obtained from the sequence of isomorphisms \eqref{Maslovsequence}.\\
Let us note from the very definition the following three properties of the Maslov index:

\begin{itemize}
\item[(i)] If $\gamma_1,\gamma_2$ are homotopic by a homotopy having endpoints in $\Lambda_0(L_0)$, then

\[\mu_{Mas}(\gamma_1,L_0)=\mu_{Mas}(\gamma_2,L_0).\]
\item[(ii)] If $\gamma_1, \gamma_2$ are such that $\gamma_1(1)=\gamma_2(0)$, then 

\[\mu_{Mas}(\gamma_1\ast\gamma_2,L_0)=\mu_{Mas}(\gamma_1,L_0)+\mu_{Mas}(\gamma_2,L_0).\]
\item[(iii)] If $\gamma(\lambda)\in\Lambda_0(L_0)$ for all $\lambda\in I$, then $\mu_{Mas}(\gamma,L_0)=0$.
\end{itemize}
Let us point out that (iii) also follows from (i) and (ii) independently of the construction.\\
The Maslov index can easily be generalised to a pair of paths in $\Lambda(n)$. To this aim let us call a pair of paths $(\gamma_1,\gamma_2)$ \textit{admissible} if

\[\gamma_1(0)\cap\gamma_2(0)=\gamma_1(1)\cap\gamma_2(1)=\{0\}.\]
In what follows we consider $\mathbb{R}^{2n}\times\mathbb{R}^{2n}$ as a symplectic space with respect to the symplectic form $(-\omega_0)\times\omega_0$. Note that the diagonal $\Delta$ is in $\Lambda(2n)$, as well as $L_1\times L_2$ for any $L_1,L_2\in\Lambda(n)$. Moreover, $L_1\cap L_2\neq\{0\}$ if and only if $(L_1\times L_2)\cap\Delta\neq\{0\}$. Hence it is natural to define the \textit{Maslov index} for a pair $(\gamma_1,\gamma_2)$ of admissible paths in $\Lambda(n)$ as

\[\mu_{Mas}(\gamma_1,\gamma_2)=\mu_{Mas}(\gamma_1\times\gamma_2,\Delta).\]
Note that the basic properties which we previously mentioned carry over immediately, i.e.,

\begin{itemize}
\item[(i')] $\mu_{Mas}(\gamma_1,\gamma_2)=0$ if $\gamma_1(\lambda)\cap\gamma_2(\lambda)=\{0\}$ for all $\lambda\in I$.
\item[(ii')] $\mu_{Mas}(\gamma_1\ast\gamma_3,\gamma_2\ast\gamma_4)=\mu_{Mas}(\gamma_1,\gamma_2)+\mu_{Mas}(\gamma_3,\gamma_4)$ if $\gamma_{1}(1)=\gamma_{3}(0)$ and $\gamma_{2}(1)=\gamma_{4}(0)$.
\item[(iii')] $\mu_{Mas}(\gamma_1,\gamma_2)=\mu_{Mas}(\gamma_3,\gamma_4)$ if $\gamma_1\simeq\gamma_3$ and $\gamma_2\simeq\gamma_4$ are homotopic by a homotopy through admissible pairs.
\end{itemize}
Also, it is not difficult to see from the construction of the Maslov index that

\begin{itemize}
\item[(iv')] $\mu_{Mas}(\gamma_1,\gamma_2)=\mu_{Mas}(\gamma_1,L_0)$ in case that $\gamma_2(\lambda)=L_0$ for some $L_0\in\Lambda(n)$ and all $\lambda\in I$,
\item[(v')] $\mu_{Mas}(\gamma_1,\gamma_2)=-\mu_{Mas}(\gamma_2,\gamma_1)$ for any admissible pair $(\gamma_1,\gamma_2)$. 
\end{itemize}
Finally, let us define the Maslov index for a non-admissible pair of paths. It is important to note that in this case there are different definitions in the literature. Here we follow \cite{Cappell}, and note that given $L_1, L_2\in\Lambda(n)$ such that $L_1\cap L_2\neq\{0\}$, there is $\varepsilon>0$ such that $e^{\Theta J}L_2\in\Lambda(n)$ and $L_1\cap e^{\Theta J}L_2=\{0\}$ for all $0<|\Theta|\leq\varepsilon$. We define the Maslov index as

\[\mu_{Mas}(\gamma_1,\gamma_2)=\mu_{Mas}(\gamma_1,e^{-\Theta J}\gamma_2),\]
where $\Theta$ is such that $\gamma_1(0)\cap e^{-\Theta' J}\gamma_2(0)=\gamma_1(1)\cap e^{-\Theta' J}\gamma_2(1)=\{0\}$ for all $0<|\Theta'|\leq\Theta$. By the homotopy invariance, it is clear that this definition does not depend on the choice of $\Theta$. Also, it coincides with the previous definition in case that the pair of paths is admissible.


\subsubsection{The Paths $\gamma_{nor}$ and $\gamma'_{nor}$}\label{section-gammanor}
The aim of this section is to compute the Maslov index for two elementary paths that will also become important in our proof of Theorem \ref{main} below. The examples also show that \eqref{Maslovsequence} is very convenient to obtain paths in $\Lambda(n)$ with a given Maslov index.\\
Let us first consider the path

\[[0,1]\ni \lambda\mapsto A(\lambda)=\diag(e^{i\pi\lambda},1,\ldots,1)\in U(n)\]  
and its projection $\overline{A}(\lambda):=A(\lambda)\cdot O(n)$ to the quotient $U(n)/O(n)$. Note that\linebreak $A(0)\diag(-1,1,\ldots,1)=A(1)$ and so $\overline{A}$ is a closed curve. Also, as $\det^2(A(\lambda))=e^{2\pi i\lambda}$, we see that the Maslov index of the corresponding path in $\Lambda(n)$ is $1$. Using the identification $\mathbb{C}^n\cong\mathbb{R}^{2n}$, it is readily seen that

\[\gamma_{nor}(\lambda):=A(\lambda)(\mathbb{R}^n\times\{0\})=\mathbb{R}(\cos(\pi\lambda)e_1+\sin(\pi\lambda)e_{n+1})+\sum^{n}_{j=2}{\mathbb{R}e_j}\in\Lambda(n).\] 
Hence we have found a path $\gamma_{nor}$ such that $\gamma_{nor}(0)=\gamma_{nor}(1)=\mathbb{R}^n\times\{0\}$ and $\mu_{Mas}(\gamma_{nor})=1$.\\
Let us now consider

\[[0,1]\ni \lambda\mapsto B(\lambda)=\diag(-ie^{ i\pi\lambda},i,\ldots,i)\in U(n)\] 
and note that again the projection $\overline{B}$ to $U(n)/O(n)$ is a closed path and $\det^2(B(\lambda))=(-1)^{n}e^{2\pi i\lambda}$. Hence

\[\gamma'_{nor}(\lambda):=B(\lambda)(\mathbb{R}^n\times\{0\})=\mathbb{R}(\sin(\pi\lambda)e_1-\cos(\pi\lambda)e_{n+1})+\sum^{2n}_{j=n+2}{\mathbb{R}e_j}\in\Lambda(n)\] 
is such that $\gamma'_{nor}(0)=\gamma'_{nor}(1)=\{0\}\times\mathbb{R}^n$ and $\mu_{Mas}(\gamma'_{nor})=1$.

\subsection{The Gap-Metric and the Spectral Flow}\label{section-spectralflow}
Our first aim of this section is to recall the definition of the gap-metric, where we follow Kato's monograph \cite{Kato}.\\
Let $H$ be a real Hilbert space and let $G(H)$ denote the set of all closed subspaces of $H$. For every $U\in G(H)$ there is a unique orthogonal projection $P_U$ onto $U$ which is a bounded operator on $H$. We set 

\[d_G(U,V)=\|P_U-P_V\|,\quad U,V\in G(H),\]
and note that this is obviously a metric on $G(H)$. The distance between two non-trivial subspaces $U,V\in G(H)$ can also be obtained as follows. Let $S_U$ denote the unit sphere in $U$ and $d(u,V)=\inf_{v\in V}{\|u-v\|}$. Then for $\delta(U,V)=\sup_{u\in S_U} d(u,V)$,

\begin{align}\label{gap}
d_G(U,V)=\max\{\delta(U,V),\delta(V,U)\},
\end{align}
which explains why $d_G(U,V)$ is called the \textit{gap} between $U$ and $V$.\\
We now consider operators $T:\mathcal{D}(T)\subset H\rightarrow H$ which we assume to be defined on a dense domain $\mathcal{D}(T)$. Let us recall that $T$ is called \textit{closed} if its graph $\gra(T)$ is closed in $H\times H$. If we denote by $\mathcal{C}(H)$ the set of all closed operators, then the gap-metric on $H\times H$ induces a metric on $\mathcal{C}(H)$ by

\[d_G(T,S)=d_G(\gra(T),\gra(S)), \quad S,T\in\mathcal{C}(H).\] 
As the adjoint of a densely defined operator is closed, every selfadjoint operator on $H$ belongs to the metric space $\mathcal{C}(H)$. Moreover, let us recall that a closed operator $T$ is called \textit{Fredholm} if its kernel and cokernel are of finite dimension. In what follows, we denote the subset of $\mathcal{C}(H)$ consisting of all $T$ which are selfadjoint and Fredholm by $\mathcal{CF}^\textup{sa}(H)$. It is well known that the spectrum $\sigma(T)$ of every selfadjoint operator is real. Moreover, if $T\in\mathcal{CF}^\textup{sa}(H)$ then $0$ is either in the resolvent set or an isolated eigenvalue of finite multiplicity (see e.g. \cite[Lemma 2.2.5]{Fredholm}).\\
It was shown in \cite{UnbSpecFlow} that for every $T\in\mathcal{CF}^\textup{sa}(H)$ there is $\varepsilon>0$ and a neighbourhood $\mathcal{N}_{T,\varepsilon}\subset\mathcal{CF}^\textup{sa}(H)$ of $T$ such that $\pm\varepsilon\notin\sigma(S)$ and the spectral projection $\chi_{[-\varepsilon,\varepsilon]}(S)$ is of finite rank for all $S\in\mathcal{N}_{T,\varepsilon}$. Let us now consider a path $\mathcal{A}=\{\mathcal{A}_\lambda\}_{\lambda\in I}$ in $\mathcal{CF}^\textup{sa}(H)$. There are $0=\lambda_0<\lambda_1<\ldots<\lambda_N=1$ such that the restriction of the path $\mathcal{A}$ to $[\lambda_{i-1},\lambda_i]$ is entirely contained in a neighbourhood $\mathcal{N}_{T_i,\varepsilon_i}$ as above for some $T_i\in\mathcal{CF}^\textup{sa}(H)$ and some $\varepsilon_i>0$. The \textit{spectral flow} of the path $\mathcal{A}$ is defined as

\begin{align}\label{sfl-def}
\sfl(\mathcal{A})=\sum^N_{i=1}{\left(\dim(\im(\chi_{[0,\varepsilon_i]}(\mathcal{A}_{\lambda_i}))-\dim(\im(\chi_{[0,\varepsilon_i]}(\mathcal{A}_{\lambda_{i-1}}))\right)}.
\end{align}
It follows by an argument of Phillips \cite{Phillips} that $\sfl(\mathcal{A})$ only depends on the path $\mathcal{A}$, and that the following fundamental property holds (see also \cite{UnbSpecFlow}).

\begin{enumerate}
\item[(i)] Let $h:I\times I\rightarrow\mathcal{CF}^\textup{sa}(H)$ be a homotopy such that the dimensions of the kernels of $h(s,0)$ and $h(s,1)$ are constant for all $s\in I$. Then 
\[\sfl(h(0,\cdot))=\sfl(h(1,\cdot)).\]
\end{enumerate}   
Moreover, it is easily seen from the definition of the spectral flow that

\begin{enumerate}
\item[(ii)] if the dimension of the kernel of $\mathcal{A}_\lambda$ is constant for all $\lambda\in I$, then $\sfl(\mathcal{A})=0$;
\item[(iii)] if $\mathcal{A}^1$ and $\mathcal{A}^2$ are two paths in $\mathcal{CF}^\textup{sa}(H)$ such that $\mathcal{A}^1_1=\mathcal{A}^2_0$, then

\[\sfl(\mathcal{A}^1\ast\mathcal{A}^2)=\sfl(\mathcal{A}^1)+\sfl(\mathcal{A}^2).\]
\end{enumerate}
Let us finally note two further elementary properties of the spectral flow which play a crucial role in our proof of Theorem \ref{main} below. The first of them has been used, e.g., in \cite[\S 7]{Pejsachowicz}.

\begin{lemma}\label{lemma-sflperturbation}
Let $\mathcal{A}:I\rightarrow\mathcal{CF}^\textup{sa}(H)$ be gap-continuous and set $\mathcal{A}^\delta=\mathcal{A}+\delta I_H$ for $\delta\in\mathbb{R}$. Then, for any sufficiently small $\delta>0$, $\mathcal{A}^\delta$ is a gap-continuous path in $\mathcal{CF}^\textup{sa}(H)$ and

\begin{align}\label{sfldelta}
\sfl(\mathcal{A})=\sfl(\mathcal{A}^\delta).
\end{align}
\end{lemma}

\begin{proof}
We note at first that the operators $\mathcal{A}^\delta_\lambda$ are selfadjoint and Fredholm for $\delta$ sufficiently small, which follows from standard stability theory (see e.g. \cite{Kato}). Moreover, the path $\mathcal{A}^\delta$ is gap-continuous by \cite[Thm. IV.2.17]{Kato}, and so $\sfl(\mathcal{A}^\delta)$ is well defined.\\ 
To show \eqref{sfldelta}, let $0=\lambda_0<\ldots<\lambda_N=1$ be a partition of the unit interval and $\varepsilon_i>0$, $i=1,\ldots,N$, for $\mathcal{A}$ as in \eqref{sfl-def}. Let $\mathcal{N}_{T,\varepsilon_i}$ be an open neighbourhood of some $T\in\mathcal{CF}^{sa}(H)$ as in the construction of the spectral flow such that $\mathcal{A}_\lambda\in\mathcal{N}_{T,\varepsilon_i}$ for all $\lambda\in[\lambda_{i-1},\lambda_i]$. Now there is $\delta_i>0$ such that $\mathcal{A}^{s\delta_i}_\lambda\in\mathcal{N}_{T,\varepsilon_i}$ for all $s\in[0,1]$ and all  $\lambda\in [\lambda_{i-1},\lambda_i]$, i.e. the spectral projections $\chi_{[-\varepsilon_i,\varepsilon_i]}(\mathcal{A}^{s\delta_i}_{\lambda})$ are of the same finite rank. Moreover, by choosing $\delta_i>0$ smaller, we can assume that 

\[\sigma(\mathcal{A}_{\lambda_i})\cap[-\delta_i,0)=\sigma(\mathcal{A}_{\lambda_{i-1}})\cap[-\delta_i,0)=\{0\}.\]
Then, as $\sigma(\mathcal{A}^{\delta_i}_\lambda)=\sigma(\mathcal{A}_\lambda)+\delta_i$, we see that

\[\dim(\im(\chi_{[0,\varepsilon_i]}(\mathcal{A}_{\lambda})))=\dim(\im(\chi_{[0,\varepsilon_i]}(\mathcal{A}^{\delta_i}_{\lambda}))),\quad \lambda=\lambda_{i-1},\lambda_i.\]
If we now set $\delta=\min\{\delta_1,\ldots,\delta_N\}>0$, then
\[\dim(\im(\chi_{[0,\varepsilon_i]}(\mathcal{A}_{\lambda})))=\dim(\im(\chi_{[0,\varepsilon_i]}(\mathcal{A}^{\delta}_{\lambda}))),\quad \lambda=\lambda_{i-1},\lambda_i\]
holds simultaneously for this $\delta$ and all $i=1,\ldots,N$, and so the assertion follows from the definition \eqref{sfl-def}. 
\end{proof}  
\noindent
Finally, let us note the following stability of the spectral flow under conjugation by invertible operators, where we denote by $M^T$ the adjoint of an operator in the real Hilbert space $H$.

\begin{lemma}\label{lemma-conjuagtion}
Let $\mathcal{A}:I\rightarrow\mathcal{CF}^\textup{sa}(H)$ be a gap-continuous path and $M:I\rightarrow GL(H)$ a continuous family of bounded invertible operators. Then $\{M^T_\lambda\mathcal{A}_\lambda M_\lambda\}_{\lambda\in I}$ is gap-continuous and

\[\sfl(M^T\mathcal{A}M)=\sfl(\mathcal{A}).\]
\end{lemma}

\begin{proof}
Note that

\begin{align*}
\gra(M^T_\lambda\mathcal{A}_\lambda M_\lambda)&=\{(u,M^T_\lambda\mathcal{A}_\lambda M_\lambda u):\, u\in M^{-1}_\lambda(\mathcal{D}(\mathcal{A}_\lambda))\}=\{(M^{-1}_\lambda v,M^T_\lambda\mathcal{A}_\lambda v):\,v\in\mathcal{D}(\mathcal{A}_\lambda)\}\\
&=\begin{pmatrix}
M^{-1}_\lambda&0\\
0&M^T_\lambda
\end{pmatrix}\,\gra(\mathcal{A}_\lambda)=:N_\lambda\gra(\mathcal{A}_\lambda)\subset H\times H,
\end{align*}
and so $\{N_\lambda P_{\gra(\mathcal{A}_\lambda)}N^{-1}_\lambda\}_{\lambda\in I}$ is a continuous family of oblique projections onto\linebreak $\{\gra(M^T_\lambda\mathcal{A}_\lambda M_\lambda)\}_{\lambda\in I}$ in $\mathcal{L}(H\times H)$. By \cite[Thm. I.6.35]{Kato}, we have for the corresponding orthogonal projections $P_{\gra(M^T_\lambda\mathcal{A}_\lambda M_\lambda)}$ onto $\gra(M^T_\lambda\mathcal{A}_\lambda M_\lambda)$ the inequality

\[\|P_{\gra(M^T_\mu\mathcal{A}_\mu M_\mu)}-P_{\gra(M^T_\lambda\mathcal{A}_\lambda M_\lambda)}\|\leq \|N_\mu P_{\gra(\mathcal{A}_\mu)}N^{-1}_\mu-N_\lambda P_{\gra(\mathcal{A}_\lambda)}N^{-1}_\lambda\|,\quad \mu,\lambda\in I.\]
Consequently, $\{P_{\gra(M^T_\lambda\mathcal{A}_\lambda M_\lambda)}\}_{\lambda\in I}$ is continuous, which shows that $M^T\mathcal{A}M$ is gap-continuous.\\ 
For the equality of the spectral flows, we just need to note that $M$ is homotopic inside $GL(H)$ to the constant path given by the identity $I_H$. Let us point out that this does not even require Kuiper's Theorem as we just need to shrink $M$ to a constant path and use that $GL(H)$ is connected. As the conjugation preserves kernel dimensions, we obtain by the homotopy invariance (i) from above

\begin{align*}
\sfl(M^T\mathcal{A}M)=\sfl(\mathcal{A}).
\end{align*}
\end{proof}


\section{Proof of Theorem \ref{main}}\label{section-proof}
The proof of Theorem \ref{main} falls naturally into two parts. In the first part we deal with the continuity of families of the type \eqref{ops-def}, where we actually consider a slightly more general setting. In the second part we show the spectral flow formula in Theorem \ref{main}.

\subsection{Continuity}
To simplify notation, we set $E=L^2(I,\mathbb{R}^{2n})$ and $H=H^1(I,\mathbb{R}^{2n})$. The aim of this step is to prove the following proposition, which we will later apply in the cases $X=I$ and $X=I\times I$.

\begin{prop}\label{prop-continuity}
Let $X$ be a metric space and $\gamma_1,\gamma_2:X\rightarrow\Lambda(n)$ two families of Lagrangian subspaces in $\mathbb{R}^{2n}$. Then

\[\mathcal{A}:X\rightarrow \mathcal{CF}^\textup{sa}(E),\quad (\mathcal{A}_\lambda u)(t)=Ju'(t),\]
where
\[\mathcal{D}(\mathcal{A}_\lambda)=\{u\in H: u(0)\in\gamma_1(\lambda), u(1)\in\gamma_2(\lambda)\},\]
is continuous with respect to the gap-metric on $\mathcal{CF}^\textup{sa}(E)$. 
\end{prop}
\noindent
We want to use \eqref{gap} and consider 

\[\delta(\gra(\mathcal{A}_\lambda),\gra(\mathcal{A}_{\lambda_0})).\]
Note at first that for $u\in\mathcal{D}(\mathcal{A}_\lambda)$ and $v\in\mathcal{D}(\mathcal{A}_{\lambda_0})$

\begin{align}\label{gapcontI}
\begin{split}
\|(u,\mathcal{A}_\lambda u)-(v,\mathcal{A}_{\lambda_0}v)\|_{E\oplus E}&=\|(u-v,J(u'-v'))\|_{E\oplus E}\\
&\leq \left(\|u-v\|^2_{E}+\|J\|\|u'-v'\|^2_{E}\right)^\frac{1}{2}\\
&=\|u-v\|_{H},
\end{split}
\end{align}
where we have used that $\|J\|=1$. Let us recall that the topology of $G_n(\mathbb{R}^{2n})$ is induced by the metric $d(L,M)=\|P_L-P_M\|$, where $P_L, P_M\in M(2n,\mathbb{R})$ are the orthogonal projections onto $L$ and $M$, respectively. Hence, by the continuity of $\gamma_1$ and $\gamma_2$, there are two families of orthogonal projections $\hat{P},\tilde{P}:X\rightarrow M(2n,\mathbb{R})$ such that 

\[\im(\hat{P}_\lambda)=\gamma_1(\lambda),\quad \im(\tilde{P}_\lambda)=\gamma_2(\lambda),\quad\lambda\in X.\]
We define for $w\in H$

\[(P_\lambda w)(t)=w(t)-(1-t)(I_{2n}-\hat{P}_\lambda)w(0)-t(I_{2n}-\tilde{P}_\lambda)w(1).\]
It is easily seen that $P^2_\lambda w=P_\lambda w$, as well as $P_\lambda w\in\mathcal{D}(\mathcal{A}_\lambda)$ for all $w\in H$ and $\lambda\in X$, which shows that

\begin{align}\label{gapcontII}
\inf_{v\in\mathcal{D}(\mathcal{A}_{\lambda_0})}\|u-v\|_{H}\leq\|u-P_{\lambda_0}u\|_{H}.
\end{align}
As 

\[u(t)-(P_{\lambda_0}u)(t)=(1-t)(I_{2n}-\hat{P}_{\lambda_0})u(0)+t(I_{2n}-\tilde{P}_{\lambda_0})u(1),\]
it follows for $u\in\mathcal{D}(\mathcal{A}_\lambda)$ that

\begin{align}\label{gapcontIII}
\begin{split}
\|u-P_{\lambda_0}u\|_{H}&\leq 2(\|(I_{2n}-\hat{P}_{\lambda_0})u(0)\|+\|(I_{2n}-\tilde{P}_{\lambda_0})u(1)\|)\\
&= 2(\|(I_{2n}-\hat{P}_{\lambda_0})\hat{P}_\lambda u(0)\|+\|(I_{2n}-\tilde{P}_{\lambda_0})\tilde{P}_\lambda u(1)\|)\\
&\leq 2(\|(I_{2n}-\hat{P}_{\lambda_0})\hat{P}_\lambda\|\|u(0)\|+\|(I_{2n}-\tilde{P}_{\lambda_0})\tilde{P}_\lambda\|\|u(1)\|),
\end{split}
\end{align}
where we have used that $u\in\mathcal{D}(\mathcal{A}_\lambda)$ and so $\hat{P}_\lambda u(0)=u(0)$ and $\tilde{P}_\lambda u(1)=u(1)$. Let us note that the factor $2$ appears in the previous estimate as we are dealing with the norm on $H$ and so we also need to take into account the derivatives of $u-P_{\lambda_0}u$ with respect to $t$.\\
Since the point evaluation is continuous in $H$, there is a constant $\alpha>0$ such that for $t=0$ and $t=1$

\begin{align}\label{gapcontIV}
\|u(t)\|\leq\alpha\|u\|_{H}=\alpha\left(\|u\|^2_{E}+\|u'\|^2_{E}\right)^\frac{1}{2}=\alpha\left(\|u\|^2_{E}+\|Ju'\|^2_{E}\right)^\frac{1}{2},
\end{align}
where we use that $J$ is an isometry on $\mathbb{R}^{2n}$. Hence, by \eqref{gapcontI}--\eqref{gapcontIV},

\begin{align*}
d((u,\mathcal{A}_\lambda u),\gra(\mathcal{A}_{\lambda_0}))&=\inf_{v\in\mathcal{D}(\mathcal{A}_{\lambda_0})}\|(u,\mathcal{A}_\lambda u)-(v,\mathcal{A}_{\lambda_0}v)\|_{E\oplus E}\\
&\leq \inf_{v\in\mathcal{D}(\mathcal{A}_{\lambda_0})}\|u-v\|_{H}\leq\|u-P_{\lambda_0}u\|_{H}\\
&\leq 2(\|(I_{2n}-\hat{P}_{\lambda_0})\hat{P}_\lambda\|\|u(0)\|+\|(I_{2n}-\tilde{P}_{\lambda_0})\tilde{P}_\lambda\|\|u(1)\|)\\
&\leq 2\alpha(\|(I_{2n}-\hat{P}_{\lambda_0})\hat{P}_\lambda\|+\|(I_{2n}-\tilde{P}_{\lambda_0})\tilde{P}_\lambda\|)(\|u\|^2_{E}+\|Ju'\|^2_{E})^\frac{1}{2}. 
\end{align*}
As the unit sphere in $\gra(\mathcal{A}_\lambda)$ is given by 

\[\{(u,\mathcal{A}_\lambda u):\, u\in\mathcal{D}(\mathcal{A}_\lambda),\, \|u\|^2_{E}+\|Ju'\|^2_{E}=1\},\]
we finally get

\begin{align}\label{gapfinalI}
\begin{split}
\delta(\gra(\mathcal{A}_\lambda),\gra(\mathcal{A}_{\lambda_0}))&=\sup\{d((u,\mathcal{A}_\lambda u),\gra(\mathcal{A}_{\lambda_0})):\,u\in\mathcal{D}(\mathcal{A}_\lambda),\,\|u\|^2+\|Ju'\|^2=1\}\\
&\leq 2\alpha(\|(I_{2n}-\hat{P}_{\lambda_0})\hat{P}_\lambda\|+\|(I_{2n}-\tilde{P}_{\lambda_0})\tilde{P}_\lambda\|).
\end{split}
\end{align}
Note that if we swap $\lambda$ and $\lambda_0$ and repeat the above argument, we also have

\begin{align}\label{gapfinalII}
\delta(\gra(\mathcal{A}_{\lambda_0}),\gra(\mathcal{A}_{\lambda}))\leq 2\alpha(\|(I_{2n}-\hat{P}_{\lambda})\hat{P}_{\lambda_0}\|+\|(I_{2n}-\tilde{P}_{\lambda})\tilde{P}_{\lambda_0}\|).
\end{align}
To finish the proof, we need the following well-known theorem that can be found, e.g., in \cite[I.6.34]{Kato}.

\begin{theorem}
Let $E$ be a Hilbert space and $P,Q$ orthogonal projections in $E$. If 

\[\|(I_E-P)Q\|<1\,\,\text{and  } \|(I_E-Q)P\|<1,\]
then 

\[\|(I_E-P)Q\|=\|(I_E-Q)P\|=\|P-Q\|.\]
\end{theorem}
\noindent
Now, as $(I_{2n}-\hat{P}_{\lambda})\hat{P}_{\lambda_0}=(I_{2n}-\hat{P}_{\lambda_0})\hat{P}_{\lambda}=0$ for $\lambda=\lambda_0$, we have for all $\lambda$ in a neighbourhood of $\lambda_0$

\[\|(I_{2n}-\hat{P}_{\lambda})\hat{P}_{\lambda_0}\|=\|(I_{2n}-\hat{P}_{\lambda_0})\hat{P}_{\lambda}\|=\|\hat{P}_{\lambda}-\hat{P}_{\lambda_0}\|\]
and likewise
\[\|(I_{2n}-\tilde{P}_{\lambda})\tilde{P}_{\lambda_0}\|=\|(I_{2n}-\tilde{P}_{\lambda_0})\tilde{P}_{\lambda}\|=\|\tilde{P}_\lambda-\tilde{P}_{\lambda_0}\|.\]
Consequently, we obtain from \eqref{gap}, \eqref{gapfinalI} and \eqref{gapfinalII} for all $\lambda$ sufficiently close to $\lambda_0$

\begin{align*}
d_G(\mathcal{A}_\lambda,\mathcal{A}_{\lambda_0})&=\max\{\delta(\gra(\mathcal{A}_\lambda),\gra(\mathcal{A}_{\lambda_0})),\delta(\gra(\mathcal{A}_{\lambda_0}),\gra(\mathcal{A}_{\lambda}))\}\\
&\leq 2\alpha(\|\hat{P}_\lambda-\hat{P}_{\lambda_0}\|+\|\tilde{P}_\lambda-\tilde{P}_{\lambda_0}\|),
\end{align*}
which shows that $\mathcal{A}=\{\mathcal{A}_\lambda\}_{\lambda\in X}$ is indeed continuous in $\mathcal{CF}(E)$. Hence Proposition \ref{prop-continuity} is shown.


\subsection{The Spectral Flow Formula}
We now prove the spectral flow formula in Theorem \ref{main} in two steps.

\subsection*{Step 1: Theorem \ref{main} for admissible paths}
We begin this first step of our proof with the following elementary observation.

\begin{lemma}\label{lemma-connected}
The set of all transversal pairs in $\Lambda(n)$, i.e. 

\begin{align}\label{set}
\{(L_1,L_2)\in\Lambda(n)\times\Lambda(n):\, L_1\cap L_2=\{0\}\}\subset\Lambda(n)\times\Lambda(n),
\end{align}
is path-connected.
\end{lemma}

\begin{proof}
Let us first recall the well-known fact that $\Lambda_0(L_0)$ is contractible, and hence path-connected, for any $L_0\in\Lambda(n)$ (see \cite[Rem. 2.5.3]{PiccioneBook}). Now let $(L_1,L_2)$ and $(L_3,L_4)$ be two transversal pairs. As in the construction of the Maslov index in Section \ref{section-Maslov}, $L'_1=e^{\Theta J}L_1$ is transversal to $L_2$ and $L_4$ for any sufficiently small $\Theta>0$. In particular, we obtain a path connecting $(L_1,L_2)$ and $(L'_1,L_2)$ inside \eqref{set}. Also, as $\Lambda_0(L'_1)$ is path-connected, there is a path connecting $(L'_1,L_2)$ and $(L'_1,L_4)$ inside \eqref{set}. Finally, there is a path from $(L'_1,L_4)$ to $(L_3,L_4)$ inside \eqref{set} as $\Lambda_0(L_4)$ is path-connected.     
\end{proof}
\noindent
This step of the proof is based on the following proposition in which we denote by $\Omega^2$ the set of all admissible pairs of paths in $\Lambda(n)$ (see \eqref{introduction-transversal}). Let us note that by Section \ref{section-gammanor} and (v') in Section \ref{section-Maslov}, $\mu_{Mas}(\gamma_{nor},L_1)=1$ and $\mu_{Mas}(L_0,\gamma'_{nor})=-1$, where $L_0=\mathbb{R}^n\times\{0\}$ and $L_1=\{0\}\times\mathbb{R}^n$. 

\begin{prop}\label{prop-Maslov}
Let 

\[\mu:\Omega^2\rightarrow\mathbb{Z}\]
be a map such that the same properties (i')-(iii') from Section \ref{section-Maslov} are satisfied, as well as

\begin{itemize}
\item[(N)] $\mu(\gamma_{nor},L_1)=1$ and $\mu(L_0,\gamma'_{nor})=-1$, where $L_0=\mathbb{R}^n\times\{0\}$ and $L_1=\{0\}\times\mathbb{R}^n$. 
\end{itemize}
Then $\mu=\mu_{Mas}$ on $\Omega^2$.
\end{prop} 
 
\begin{proof}

We note at first that we have by the properties (ii') and (iii') homomorphisms

\begin{align}\label{homomorphisms}
\mu, \mu_{Mas}:\pi_1(\Lambda(n)\times\Lambda(n),(L_0,L_1))\rightarrow\mathbb{Z}
\end{align}
and we now claim that they coincide.\\
We first note that 

\[\pi_1(\Lambda(n)\times\Lambda(n),(L_0,L_1))\cong\pi_1(\Lambda(n),L_0)\times\pi_1(\Lambda(n),L_1)\cong\mathbb{Z}\oplus\mathbb{Z},\]
where the first isomorphism is induced by the projections onto the components and the second one is given by the Maslov index. As $\gamma_{nor}(0)=L_0$, $\gamma'_{nor}(0)=L_1$ and $\mu_{Mas}(\gamma_{nor})=\mu_{Mas}(\gamma'_{nor})=1$, we see that the pairs of paths

\[\{(\gamma_{nor},L_1), (L_0,\gamma'_{nor})\}\]
define a basis of $\pi_1(\Lambda(n)\times\Lambda(n),(L_0,L_1))$. Since the homomorphisms in \eqref{homomorphisms} coincide on this basis by (N), it follows that $\mu$ and $\mu_{Mas}$ are indeed equal for closed paths based at $(L_0,L_1)$.\\ 
Let us now assume that $(\gamma_1,\gamma_2)\in\Omega^2$ is an arbitrary admissible pair of paths. We connect $(L_0,L_1)$ to $(\gamma_1(0),\gamma_2(0))$ by a pair of paths $(\gamma_3,\gamma_4)$ and $(\gamma_1(1),\gamma_2(1))$ to $(L_0,L_1)$ by a pair of paths $(\gamma_5,\gamma_6)$, where we can assume by Lemma \ref{lemma-connected} that $\gamma_3(\lambda)\cap\gamma_4(\lambda)=\gamma_5(\lambda)\cap\gamma_6(\lambda)=\{0\}$ for all $\lambda$. Then by (i'), (ii') and the first step of our proof

\begin{align*}
\mu(\gamma_1,\gamma_2)&=\mu(\gamma_3,\gamma_4)+\mu(\gamma_1,\gamma_2)+\mu(\gamma_5,\gamma_6)\\
&=\mu((\gamma_3,\gamma_4)\ast(\gamma_1,\gamma_2)\ast(\gamma_5,\gamma_6))=\mu_{Mas}((\gamma_3,\gamma_4)\ast(\gamma_1,\gamma_2)\ast(\gamma_5,\gamma_6))\\
&=\mu_{Mas}(\gamma_3,\gamma_4)+\mu_{Mas}(\gamma_1,\gamma_2)+\mu_{Mas}(\gamma_5,\gamma_6)=\mu_{Mas}(\gamma_1,\gamma_2),
\end{align*}  
which proves the proposition.
\end{proof} 
\noindent
\begin{rem}
Let $(\gamma_1,\gamma_2)$ be a pair of paths in $\Lambda(n)$ as in (i'), i.e. $\gamma_1(\lambda)\cap\gamma_2(\lambda)=\{0\}$ for all $\lambda\in I$. Then $(\gamma_1,\gamma_2)$ is homotopic to the constant pair of paths $(\widetilde{\gamma}_1(\lambda),\widetilde{\gamma}_2(\lambda))=(\gamma_1(0),\gamma_2(0))$, $\lambda\in I$, by a homotopy of admissible pairs. Hence $\mu(\gamma_1,\gamma_2)=\mu(\widetilde{\gamma}_1,\widetilde{\gamma}_2)$ by (iii'). As 

\[\mu(\widetilde{\gamma}_1,\widetilde{\gamma}_2)=\mu((\widetilde{\gamma}_1,\widetilde{\gamma}_2)\ast(\widetilde{\gamma}_1,\widetilde{\gamma}_2))=\mu(\widetilde{\gamma}_1,\widetilde{\gamma}_2)+\mu(\widetilde{\gamma}_1,\widetilde{\gamma}_2)\]
by (ii'), we see that $\mu(\gamma_1,\gamma_2)=\mu(\widetilde{\gamma}_1,\widetilde{\gamma}_2)=0$ and so (i') follows from (ii') and (iii'). Consequently, Proposition \ref{prop-Maslov} actually characterises the Maslov index by the three axioms (ii'), (iii') and (N).
\end{rem}
\noindent
We now define 

\[\mu:\Omega^2\rightarrow\mathbb{Z},\quad \mu(\gamma_1,\gamma_2)=\sfl(\mathcal{A}),\]
where $\mathcal{A}$ is the path of differential operators \eqref{ops-def} for the pair $(\gamma_1,\gamma_2)$. We aim to use Proposition \ref{prop-Maslov} to show Theorem \ref{main} and so we need to check the properties (i'), (ii'), (iii') and (N). Let us first note that (i') follows immediately from (ii) in Section \ref{section-spectralflow} and the fact that $\ker(\mathcal{A}_\lambda)=\gamma_1(\lambda)\cap\gamma_2(\lambda)$. Also, (ii') follows from (iii) in Section \ref{section-spectralflow}. Finally, (ii') is an immediate consequence of the homotopy invariance (i) of the spectral flow and Proposition \ref{prop-continuity}.\\
Hence it remains to show that $\mu(\gamma_{nor},L_1)=1$ and $\mu(L_0,\gamma'_{nor})=-1$, which will be a direct consequence of the following lemma.

\begin{lemma}
The spectra of the operators $\mathcal{A}_\lambda$ in \eqref{ops-def} are

\begin{itemize}
\item[(i)] for $(\gamma_1,\gamma_2)=(\gamma_{nor},L_1)$
\[\sigma(\mathcal{A}_\lambda)=\left\{\pi\lambda-\frac{\pi}{2}+\pi k:\,k\in\mathbb{Z}\right\}\cup\left\{\frac{\pi}{2}+k\pi:\,k\in\mathbb{Z} \right\},\]
\item[(ii)] for $(\gamma_1,\gamma_2)=(L_0,\gamma'_{nor})$

\[\sigma(\mathcal{A}_\lambda)=\left\{-\pi\lambda+\frac{\pi}{2}+\pi k:\,k\in\mathbb{Z}\right\}\cup\left\{\frac{\pi}{2}+k\pi:\,k\in\mathbb{Z}\right\}.\]
\end{itemize}

\end{lemma}

\begin{proof}
We consider $Ju'=\mu u$ and note that the solutions of this equation are

\[u(t)=\exp(-\mu tJ)v,\quad t\in I, v\in\mathbb{R}^{2n}.\]
Let us first consider the path in (i). Then $u$ belongs to the domain of  $\mathcal{A}_\lambda$ if and only if 

\begin{align}\label{equ}
u(0)=v\in\gamma_{nor}(\lambda),\quad u(1)=\exp(-\mu J)v\in L_1.
\end{align}
As $\exp(-\mu J)v\in L_1$ if and only if $v\in\exp(\mu J)L_1$, and $\exp(\mu J)=\cos(\mu)I_{2n}+\sin(\mu)J$, we see that \eqref{equ} is equivalent to

\[(\cos(\mu)I_{2n}+\sin(\mu)J)(\{0\}\times\mathbb{R}^n)\cap\left(\mathbb{R}(\cos(\pi\lambda)e_1+\sin(\pi\lambda)e_{n+1})+\sum^{n}_{j=2}{\mathbb{R}e_j}\right)\neq\{0\}.\]
There are two different cases where these spaces intersect non-trivially. Firstly, if $\cos(\mu)=0$, i.e. $\mu=\frac{\pi}{2}+k\pi$ for $k\in\mathbb{Z}$. Secondly, if there is an $\alpha\neq 0$ such that $\sin(\pi\lambda)e_{n+1}=\alpha\cos(\mu)e_{n+1}$ and $\cos(\pi\lambda)e_1=-\alpha\sin(\mu)e_1$, where we use that $Je_{n+1}=-e_1$. Of course, the latter equations are equivalent to  $\sin(\pi\lambda)=\alpha\cos(\mu)$ and $\cos(\pi\lambda)=-\alpha\sin(\mu)$, which can be rewritten as

\begin{align*}
e^{i\pi\lambda}&=\cos(\pi\lambda)+i\sin(\pi\lambda)=\alpha (-\sin(\mu)+i\cos(\mu))=\alpha i e^{i\mu}=\alpha e^{i(\mu+\frac{\pi}{2})}.
\end{align*} 
Hence $|\alpha|=1$, and this equation holds if and only if $\pi\lambda=\mu+\frac{\pi}{2}+k\pi$, or equivalently $\mu=\pi\lambda-\frac{\pi}{2}-k\pi$.\\
In (ii), $u$ belongs to the domain of  $\mathcal{A}_\lambda$ if and only if 

\begin{align*}
u(0)=v\in L_0,\quad u(1)=\exp(-\mu J)v\in\gamma'_{nor},
\end{align*}
which is equivalent to

\[(\cos(\mu)I_{2n}-\sin(\mu)J)(\mathbb{R}^n\times\{0\})\cap\left(\mathbb{R}(\sin(\pi\lambda)e_1-\cos(\pi\lambda)e_{n+1})+\sum^{2n}_{j=n+2}{\mathbb{R}e_j}\right)\neq\{0\}.\]
Again, there are two cases in which this intersection is non-trivial. Firstly, $\mu=k\pi+\frac{\pi}{2}$ where $\cos(\mu)=0$. Secondly, if there is some $\alpha\neq 0$ such that $\alpha\cos(\mu)e_1=\sin(\pi\lambda)e_1$ and $\alpha\sin(\mu)e_{n+1}=\cos(\pi\lambda)e_{n+1}$, which is equivalent to

\[e^{i\pi\lambda}=\cos(\pi\lambda)+i\sin(\pi\lambda)=\alpha(sin(\mu)+i\cos(\mu))=\alpha i e^{-i\mu}=\alpha e^{i(\frac{\pi}{2}-\mu)}.\] 
Hence $|\alpha|=1$, and the latter equation holds if and only if $\pi\lambda=\frac{\pi}{2}-\mu+k\pi$ which finally shows that $\mu=-\pi\lambda+\frac{\pi}{2}+k\pi$.
\end{proof}
\noindent
We see from the previous lemma that in both cases there is only one eigenvalue of $\mathcal{A}_\lambda$ that crosses the axis whilst the parameter $\lambda$ travels from $0$ to $1$. It is now an immediate consequence of the definition of the spectral flow that $\sfl(\mathcal{A})=1$ for $(\gamma_1,\gamma_2)=(\gamma_{nor},L_1)$ and $\sfl(\mathcal{A})=-1$ for $(\gamma_1,\gamma_2)=(L_0,\gamma'_{nor})$. Hence Theorem \ref{main} is shown in the admissible case.


\subsection*{Step 2: The general case}
Let $(\gamma_1,\gamma_2)$ be a pair of paths in $\Lambda(n)$ which is not necessarily admissible, and let $\mathcal{A}$ be the path \eqref{ops-def}. Let $\delta>0$ be as in Lemma \ref{lemma-sflperturbation} such that 

\begin{align*}
\sfl(\mathcal{A})=\sfl(\mathcal{A}^{\delta_0})
\end{align*} 
for all $0\leq \delta_0\leq\delta$.\\
We consider the solution $\Psi:I\rightarrow\Sp(2n,\mathbb{R})$ of the differential equation

\begin{align*}
\begin{cases}
J\Psi'(t)+\delta_0\Psi(t)=0,\quad t\in I\\
\Psi(0)=I_{2n},
\end{cases}
\end{align*}
and the operator $M\in GL(L^2(I,\mathbb{R}^{2n}))$ given by $(Mu)(t)=\Psi(t)u(t)$, $t\in I$. Then, as $\mathcal{D}(\mathcal{A}^{\delta_0}_\lambda)=\mathcal{D}(\mathcal{A}_\lambda)$, $M^T\mathcal{A}^{\delta_0}_\lambda M$ is defined on the domain

\begin{align*}
\mathcal{D}(M^T\mathcal{A}^{\delta_0}_\lambda M)&=M^{-1}(\mathcal{D}(\mathcal{A}^{\delta_0}_\lambda))=\{\Psi(\cdot)^{-1}u\in H^1(I,\mathbb{R}^{2n}):u(0)\in\gamma_1(\lambda),\, u(1)\in\gamma_2(\lambda)\}\\
&=\{v\in H^1(I,\mathbb{R}^{2n}):\, v(0)\in\gamma_1(\lambda), v(1)\in\Psi(1)^{-1}\gamma_2(\lambda)\}
\end{align*}
and given by

\begin{align*}
(M^T\mathcal{A}^{\delta_0}_\lambda Mu)(t)&=M^T(J\Psi'(t)u(t)+J\Psi(t)u'(t)+\delta_0\Psi(t)u(t))\\
&=M^T(-\delta_0\Psi(t)u(t)+J\Psi(t)u'(t)+\delta_0\Psi(t)u(t))=\Psi(t)^TJ\Psi(t)u'(t)=Ju'(t).
\end{align*}
As $\Psi(t)=\exp(\delta_0 Jt)$, $t\in I$, we see that $\Psi(1)^{-1}=\exp(-\delta_0 J)$. Finally, if $\delta_0>0$ is sufficiently small, we obtain by Step 1, Proposition \ref{prop-continuity} and the definition of the Maslov index for non-admissible paths in Section \ref{section-Maslov},

\[\sfl(\mathcal{A})=\sfl(\mathcal{A}^{\delta_0})=\mu_{Mas}(\gamma_1,e^{-\delta_0 J}\gamma_2)=\mu_{Mas}(\gamma_1,\gamma_2),\]
which proves Theorem \ref{main} in the general case.


\section{A Spectral Flow Formula for Hamiltonian Systems}


Let $\gamma_1,\gamma_2:I\rightarrow\Lambda(n)$ be two paths of Lagrangian subspaces in $\mathbb{R}^{2n}$. We note for later reference the following two standard properties of the Maslov index (see e.g. \cite{RobbinMaslov})

\begin{enumerate}
\item[(vi')] If $\Psi:I\rightarrow\Sp(2n,\mathbb{R})$ is a path of symplectic matrices, then

\begin{align}\label{sympinv}
\mu_{Mas}(\Psi\gamma_1,\Psi \gamma_2)=\mu_{Mas}(\gamma_1,\gamma_2).
\end{align}
\item[(vii')] If $\gamma'_1,\gamma'_2:I\rightarrow\Lambda(n)$ denote the reverse paths defined by $\gamma'_1(\lambda)=\gamma_1(1-\lambda)$ and $\gamma'_2(\lambda)=\gamma_2(1-\lambda)$, then

\begin{align}\label{reverse}
\mu_{Mas}(\gamma'_1,\gamma'_2)=-\mu_{Mas}(\gamma_1,\gamma_2).
\end{align}
\end{enumerate}
Moreover, we need below the following homotopy invariance property which is an immediate consequence of (iii') in Section \ref{section-Maslov} and the definition of the Maslov index for non-admissible pairs of paths:

\begin{enumerate}
\item[(viii')] $\mu_{Mas}(\gamma_1,\gamma_2)=\mu_{Mas}(\gamma_3,\gamma_4)$ if $\gamma_1\simeq\gamma_3$ and $\gamma_2\simeq\gamma_4$ are homotopic by homotopies with fixed endpoints.
\end{enumerate}
Let now $S:I\times I\rightarrow M(2n,\mathbb{R})$ be a two parameter family of symmetric matrices and let us consider

\begin{equation}\label{Hamiltonian}
\left\{
\begin{aligned}
Ju'(t)+S_\lambda(t)u(t)&=0,\quad t\in I\\
(u(0),u(1))\in \gamma_1(\lambda)&\times \gamma_2(\lambda),
\end{aligned}
\right.
\end{equation}
as well as the differential operators

\begin{align}\label{A-hamiltonian}
\mathcal{A}_\lambda:\mathcal{D}(\mathcal{A}_\lambda)\subset L^2(I,\mathbb{R}^{2n})\rightarrow L^2(I,\mathbb{R}^{2n}),\quad (\mathcal{A}_\lambda u)(t)=Ju'(t)+S_\lambda(t)u(t) 
\end{align}
on the domains $\mathcal{D}(\mathcal{A}_\lambda)=\{u\in H^1(I,\mathbb{R}^{2n}):\, u(0)\in \gamma_1(\lambda),\, u(1)\in \gamma_2(\lambda)\}$.\\
We denote for $\lambda\in I$ by $\Psi_\lambda:I\rightarrow\Sp(2n,\mathbb{R})$ the matrices defined by

\begin{equation}\label{Psi}
\left\{
\begin{aligned}
J\Psi'_\lambda(t)+S_\lambda(t)\Psi_\lambda(t)&=0,\quad t\in I\\
\Psi_\lambda(0)&=I_{2n},
\end{aligned}
\right.
\end{equation}
and we set $(\Psi\gamma_1)(\lambda)=\Psi_\lambda(1) \gamma_1(\lambda)$. The aim of this final section is to obtain the following spectral flow formula from Theorem \ref{main}.

\begin{theorem}\label{thm-Hamiltonian}
Under the assumptions above, $\mathcal{A}$ is a gap-continuous path of selfadjoint Fredholm operators on $L^2(I,\mathbb{R}^{2n})$ and

\[\sfl(\mathcal{A})=\mu_{Mas}(\Psi\gamma_1,\gamma_2).\]
\end{theorem}

\begin{proof}
We define a continuous family of bounded invertible operators on $L^2(I,\mathbb{R}^{2n})$ by $(M_\lambda u)(t)=\Psi_\lambda(t) u(t)$, $t\in I$. Then 

\begin{align*}
(M^T_\lambda\mathcal{A}_\lambda M_\lambda u)(t)&= \Psi^T_\lambda(t)(J\Psi'_\lambda(t)u(t)+J\Psi_\lambda(t)u'(t)+S_\lambda(t)\Psi_\lambda(t)u(t))\\
&=\Psi^T_\lambda(t)(-S_\lambda(t)\Psi_\lambda(t)u(t))+\Psi^T_\lambda(t)J\Psi_\lambda(t)u'(t)+\Psi^T_\lambda(t)S_\lambda(t)\Psi_\lambda(t)u(t)\\
&=Ju'(t)
\end{align*}
and

\begin{align*}
\mathcal{D}(M^T_\lambda\mathcal{A}_\lambda M_\lambda)&= M^{-1}_\lambda(\mathcal{D}(\mathcal{A}_\lambda))=\{u\in H^1(I,\mathbb{R}^{2n}):\, u(0)\in\gamma_1(\lambda), u(1)\in \Psi_\lambda(1)^{-1}\gamma_2(\lambda)\}.
\end{align*}
By Theorem \ref{main}, $M^T\mathcal{A}M$ is gap-continuous, and it follows from Lemma \ref{lemma-conjuagtion} that $\mathcal{A}$ is gap-continuous as well. Moreover, we obtain from Theorem \ref{main} and \eqref{sympinv} that

\[\sfl(\mathcal{A})=\sfl(M^T\mathcal{A}M)=\mu_{Mas}(\gamma_1,\Psi_{(\cdot)}(1)^{-1}\gamma_2)=\mu_{Mas}(\Psi\gamma_1,\gamma_2).\] 
\end{proof}
\noindent
Note that we obtain from $\Psi$ and $\gamma_1$ further pairs of paths in $\Lambda(n)$ by 

\[I\ni t\mapsto \Psi_0(t)\gamma_1(0)\in\Lambda(n),\quad I\ni t\mapsto\Psi_1(t)\gamma_1(1)\in\Lambda(n).\]
The following corollary is an easy reformulation of the previous theorem.

\begin{cor}\label{cor-Hamiltonian}
Under the previous assumptions,

\[\sfl(\mathcal{A})=\mu_{Mas}(\Psi_1(\cdot)\gamma_1(1),\gamma_2(1))+\mu_{Mas}(\gamma_1,\gamma_2)-\mu_{Mas}(\Psi_0(\cdot)\gamma_1(0),\gamma_2(0)).\]
\end{cor}

\begin{proof}
We consider the family $\Gamma:I\times I\rightarrow\Lambda(n)\times\Lambda(n)$ defined by $\Gamma(\lambda,t)=(\Psi_\lambda(t)\gamma_1(\lambda),\gamma_2(\lambda))$. We set 

\[\eta_1(t)=\Gamma(0,t), \eta_2(\lambda)=\Gamma(\lambda,1), \eta_3(t)=\Gamma(1,1-t), \eta_4(\lambda)=\Gamma(1-\lambda,0).\]
As $I\times I$ is contractible, $\eta_1\ast\eta_2\ast\eta_3\ast\eta_4$ is homotopic to a constant path by a homotopy with fixed endpoints. Hence the Maslov index of $\eta_1\ast\eta_2\ast\eta_3\ast\eta_4$ vanishes by (viii').\\  
As $\mu_{Mas}(\eta_4)=-\mu_{Mas}(\gamma_1,\gamma_2)$, $\mu_{Mas}(\eta_3)=-\mu_{Mas}(\Psi_1(\cdot)\gamma_1(1),\gamma_2(1))$ and $\Psi_\lambda(0)=I_{2n}$ for all $\lambda\in I$, it follows that

\[\mu_{Mas}(\Psi \gamma_1,\gamma_2)=-\mu_{Mas}(\Psi_0(\cdot)\gamma_1(0),\gamma_2(0))+\mu_{Mas}(\gamma_1,\gamma_2)+\mu_{Mas}(\Psi_1(\cdot)\gamma_1(1),\gamma_2(1)).\]
The corollary is now an immediate consequence of the previous theorem.
\end{proof}
\noindent
Note that if $S_0(t)=S_1(t)$ for all $t\in I$, then $\Psi_0(t)=\Psi_1(t)$, $t\in I$. If, moreover, $\gamma_1(0)=\gamma_1(1)$ and $\gamma_2(0)=\gamma_2(1)$, then we obtain from the previous corollary that

\[\sfl(\mathcal{A})=\mu_{Mas}(\gamma_1,\gamma_2).\]
Consequently, under these assumptions the paths \eqref{A-hamiltonian} and \eqref{ops-def} have the same spectral flow and so the spectral flow of \eqref{A-hamiltonian} does not depend on the family of matrices $S$. Note that each $S_\lambda$ is $\mathcal{A}_\lambda$-compact, i.e. $S_\lambda:\mathcal{D}(\mathcal{A}_\lambda)\rightarrow L^2(I,\mathbb{R}^{2n})$ is compact with respect to the graph norm of $\mathcal{A}_\lambda$ on $\mathcal{D}(\mathcal{A}_\lambda)$. Let us point out that for closed paths of bounded selfadjoint Fredholm operators, the spectral flow is invariant under perturbations by compact selfadjoint operators (see \cite[Prop. 3.8]{FPR}).\\ 
Let us now consider again the general setting of Corollary \ref{cor-Hamiltonian}, let $\alpha, \beta:[0,1]\rightarrow[0,1]$ be two continuous functions such that 

\begin{align}\label{alphabeta}
\beta(\lambda)=\alpha(\lambda)+\lambda,\quad \lambda\in[0,1].
\end{align}
Our final result generalises Theorem 2 of \cite{Hu}, where the following spectral flow formula was shown for a particular class of functions $\alpha, \beta$ that satisfy \eqref{alphabeta}.

\begin{cor}
Under the assumptions of Corollary \ref{cor-Hamiltonian},

\begin{align*}
\sfl(\mathcal{A})&=\mu_{Mas}(\Psi_0(\alpha(\cdot))\gamma_1(0),\Psi_0(\beta(\cdot))\Psi_0(1)^{-1}\gamma_2(0))+\mu_{Mas}(\gamma_1,\gamma_2)\\
&-\mu_{Mas}(\Psi_1(\alpha(\cdot))\gamma_1(1),\Psi_1(\beta(\cdot))\Psi_1(1)^{-1}\gamma_2(1)).
\end{align*}
\end{cor}  

\begin{proof}
We define maps $h_1,h_2:I\times I\rightarrow I$ by

\[h_1(s,\lambda)=(1-s)\alpha(\lambda)+s(1-\lambda),\qquad h_2(s,\lambda)=(1-s)\beta(\lambda)+s,\qquad\]
and consider for $i=1,2$ the homotopies

\[H_i:I\times I\rightarrow\Lambda(n)\times\Lambda(n),\quad H_i(s,\lambda)=(\Psi_i(h_1(s,\lambda))\gamma_1(i),\Psi_i(h_2(s,\lambda))\Psi_i(1)^{-1}\gamma_2(i)).\]
As $\alpha(0)=\beta(0)$, we see that $h_1(s,0)=h_2(s,0)$ and so

\[H_i(s,0)=(\Psi_i(h_1(s,0))\gamma_1(i),\Psi_i(h_2(s,0))\Psi_i(1)^{-1}\gamma_2(i))=(\gamma_1(i),\Psi_i(1)^{-1}\gamma_2(i))\]
is independent of $s$, where we have used (vi'). Moreover, since $\alpha(1)=0$, $\beta(1)=1$, and $\Psi_i(0)=I_{2n}$,

\[H_i(s,1)=(\Psi_i(0)\gamma_1(i),\Psi_i(1)\Psi_i(1)^{-1}\gamma_2(i))=(\gamma_1(i),\gamma_2(i)),\]
and so $H_i$ is a homotopy with fixed endpoints. Hence $\mu_{Mas}(H_i(0,\cdot))=\mu_{Mas}(H_i(1,\cdot))$ by (viii') from above.\\
Finally, we note that

\begin{align*}
H_i(0,\lambda)&=(\Psi_i(\alpha(\lambda))\gamma_1(i),\Psi_i(\beta(\lambda))\Psi_i(1)^{-1}\gamma_2(i)),\\
H_i(1,\lambda)&=(\Psi_i(1-\lambda)\gamma_1(i),\Psi_i(1)\Psi_i(1)^{-1}\gamma_2(i))=(\Psi_i(1-\lambda)\gamma_1(i),\gamma_2(i))
\end{align*}
for all $\lambda\in I$, and

\[\mu_{Mas}(\Psi_i(1-\cdot)\gamma_1(i),\gamma_2(i))=-\mu_{Mas}(\Psi_i(\cdot)\gamma_1(i),\gamma_2(i)),\]
where we have used (vii'). Now the assertion of the corollary follows from Corollary \ref{cor-Hamiltonian}. 


\end{proof}
\noindent
Finally, let us briefly point out that a version of the Morse Index Theorem in semi-Riemannian geometry from \cite{PejsachowiczGeod} can easily be derived from Theorem \ref{thm-Hamiltonian} as well. We do not intend to explain the geometric content of the theorem, but just mention that it deals with non-trivial solutions of boundary value problems of the type

\begin{equation}\label{Jacobi}
\left\{
\begin{aligned}
Ju'(t)+S_\lambda(t)u(t)&=0,\quad t\in I\\
u(0),u(1)\in\{0\}&\times\mathbb{R}^n
\end{aligned}
\right.
\end{equation}
where $J$ is as in \eqref{J} and $S_\lambda$ is again a family of symmetric $2n\times 2n$ matrices. If we consider the operators $\mathcal{A}_\lambda$ in \eqref{A-hamiltonian} for the equations \eqref{Jacobi}, then

\begin{align}\label{semiR}
\sfl(\mathcal{A})=\mu_{Mas}(\Psi(\{0\}\times\mathbb{R}^n),\{0\}\times\mathbb{R}^n)
\end{align}
by Theorem \ref{thm-Hamiltonian}, where $\Psi=\{\Psi_\lambda(1)\}_{\lambda\in I}$ is the path in $\Sp(2n,\mathbb{R})$ obtained as in \eqref{Psi}. This is Proposition 6.1 in \cite{PejsachowiczGeod}. Note that in this setting the path $\mathcal{A}=\{\mathcal{A}_\lambda\}_{\lambda\in I}$ has the constant domain $\mathcal{D}(\mathcal{A}_\lambda)=\{u\in H^1(I,\mathbb{R}^{2n}):\, u(0),u(1)\in \{0\}\times\mathbb{R}^n\}$, which allows to compute its spectral flow by crossing forms (see \cite{Robbin} and \cite{WaterstraatHomoclinics}) and yields the different proof of \eqref{semiR} given in \cite{PejsachowiczGeod}.

\vspace*{1.3cm}

\begin{minipage}{1.0\textwidth}
\begin{minipage}{0.4\textwidth}
Marek Izydorek\\
Gdansk University of Technology\\
Narutowicza 11/12\\
80-233 Gdansk\\
Poland\\
E-mail:  marek.izydorek@pg.edu.pl
\end{minipage}
\hfill
\begin{minipage}{0.4\textwidth}
Joanna Janczewska\\
Gdansk University of Technology\\
Narutowicza 11/12\\
80-233 Gdansk\\
Poland\\
E-mail: joanna.janczewska@pg.edu.pl 
\end{minipage}
\end{minipage}

\vspace{1cm}
Nils Waterstraat\\
School of Mathematics,\\
Statistics \& Actuarial Science\\
University of Kent\\
Canterbury\\
Kent CT2 7NF\\
UNITED KINGDOM\\
E-mail: n.waterstraat@kent.ac.uk

\end{document}